\documentclass[reqno]{amsart}

\input xy
\xyoption{all} 
\usepackage{epsfig}
\usepackage{color}

\usepackage{amsthm}
\usepackage{amssymb}
\usepackage{amsmath}
\usepackage{amscd}

\usepackage{amsopn}
\usepackage{url}
\usepackage{hyperref}\hypersetup{colorlinks}

\usepackage{color} 

\definecolor{darkred}{rgb}{1,0,0} 
\definecolor{darkgreen}{rgb}{0,0.8,0}
\definecolor{darkblue}{rgb}{0,0,1}

\hypersetup{colorlinks,
linkcolor=darkblue,
filecolor=darkgreen,
urlcolor=darkred,
citecolor=darkgreen}

\numberwithin{equation}{section}
\newtheorem {Theorem}{Theorem}
\numberwithin{Theorem}{section}

\newtheorem {Lemma}[Theorem]{Lemma}

\newtheorem {Proposition}[Theorem]{Proposition}

\theoremstyle{definition}

\theoremstyle{remark}
\newtheorem{Remark}[Theorem]{Remark}

\expandafter\chardef\csname pre amssym.def at\endcsname=\the\catcode`\@
\catcode`\@=11
\def\undefine#1{\let#1\undefined}
\def\newsymbol#1#2#3#4#5{\let\next@\relax
 \ifnum#2=\@ne\let\next@\msafam@\else
 \ifnum#2=\tw@\let\next@\msbfam@\fi\fi
 \mathchardef#1="#3\next@#4#5}
\def\mathhexbox@#1#2#3{\relax
 \ifmmode\mathpalette{}{\m@th\mathchar"#1#2#3}%
 \else\leavevmode\hbox{$\m@th\mathchar"#1#2#3$}\fi}
\def\hexnumber@#1{\ifcase#1 0\or 1\or 2\or 3\or 4\or 5\or 6\or 7\or 8\or
 9\or A\or B\or C\or D\or E\or F\fi}

\font\teneufm=eufm10
\font\seveneufm=eufm7
\font\fiveeufm=eufm5
\newfam\eufmfam
\textfont\eufmfam=\teneufm
\scriptfont\eufmfam=\seveneufm
\scriptscriptfont\eufmfam=\fiveeufm

\catcode`\@=\csname pre amssym.def at\endcsname

\def    \eps    {\epsilon}

\newcommand{\CH}{{\mathcal H}}
\newcommand{\CA}{{\mathcal A}}

\newcommand{\CS}{{\mathcal S}}

\newcommand{\supp}{\operatorname{supp}}

\newcommand{\const}{{\mathit const}}

\newcommand{\tQ}{\tilde{Q}}
\newcommand{\hQ}{\hat{Q}}

\newcommand{\tH}{\tilde{H}}

\def    \nat    {{\natural}}

\def    \C      {{\mathbb C}}
\def    \R      {{\mathbb R}}

\def    \Z      {{\mathbb Z}}
\def    \N      {{\mathbb N}}

\def    \CP     {{\mathbb C}{\mathbb P}}

\def    \12    {{\frac{1}{2}}}

\def    \p      {\partial}

\def    \HF     {\operatorname{HF}}

\def    \Fix     {\operatorname{Fix}}

\newcommand{\hq}{\CH_{\scriptscriptstyle{Q}}}

\newcommand{\hhq}{\hat{\CH}_{\scriptscriptstyle{Q}}}

\newcommand{\jq}{J_{\scriptscriptstyle{Q}}}

\newcommand{\jkq}{J_{\scriptscriptstyle{kQ}}}

\def    \MUCZ  {\operatorname{\mu_{\scriptscriptstyle{CZ}}}}

\begin{document}


\setlength{\smallskipamount}{6pt}
\setlength{\medskipamount}{10pt}
\setlength{\bigskipamount}{16pt}





\title[Hyperbolic Quadratic at Infinity]{Periodic Orbits of
  Hamiltonian Systems Linear and Hyperbolic at infinity }

\author[Ba\c sak G\"urel]{Ba\c sak Z. G\"urel}

\address{Department of Mathematics, University of Central Florida,
  Orlando, FL 32816, USA} \email{basak.gurel@ucf.edu}


\date{\today} 

\thanks{The work is partially supported by the NSF grants DMS-0906204
  and DMS-1207680.}


\begin{abstract} 
  We consider Hamiltonian diffeomorphisms of symplectic Euclidean
  spaces, generated by compactly supported time-dependent
  perturbations of hyperbolic quadratic forms. We prove that, under
  some natural assumptions, such a diffeomorphism must have simple
  periodic orbits of arbitrarily large period when it has fixed points
  which are not necessary from a homological perspective.
\end{abstract}

\maketitle

\tableofcontents
\section{Introduction and main results}
\label{sec:main-results}
\subsection{Introduction}
\label{sec:intro}
In this paper we consider time-dependent Hamiltonians $H$ on $\R^{2n}$
which, outside a compact set, are autonomous and coincide with a
hyperbolic quadratic form (i.e., a non-degenerate quadratic form whose
Hamiltonian vector field has no purely imaginary eigenvalues). We
prove that, under some additional conditions, the Hamiltonian
diffeomorphism $\varphi_H$ must have simple (i.e., uniterated)
periodic orbits of arbitrarily large (prime) period when it has
certain ``homologically unnecessary'' fixed points. In particular,
$\varphi_H$ then has infinitely many periodic orbits. To be more
precise, this result holds provided that $\varphi_H$ has at least one
non-degenerate (or even homologically non-trivial) fixed point with
non-zero mean index, and the quadratic form (i.e., the corresponding
linear Hamiltonian vector field) has only real eigenvalues. (See
Remark \ref{rk:complex} for the case of complex eigenvalues.)

Our main motivation for studying this question comes from a variant of
the Conley conjecture, applicable to manifolds for which the standard
Conley conjecture fails. Recall in this connection that the latter
asserts the existence of infinitely many periodic orbits for every
Hamiltonian diffeomorphism of a closed symplectic manifold.  This is
the case for manifolds with spherically-vanishing first Chern class
(of the tangent bundle) and also for negative monotone manifolds; see
\cite{CGG,GG:gaps,He:irr} and also \cite{FrHa,Gi:conley,GG:neg-mon,
  Hi,LeC,SZ}. However, the Conley conjecture, as stated, fails for
some simple manifolds such as $S^2$: an irrational rotation of $S^2$
about the $z$-axis has only two periodic orbits, which are also the
fixed points; these are the poles.  In fact, any manifold that admits
a Hamiltonian torus action with isolated fixed points also admits a
Hamiltonian diffeomorphism with finitely many periodic orbits. In
particular, $\CP^n$, the Grassmannians, and, more generally, most of
the coadjoint orbits of compact Lie groups as well as symplectic toric
manifolds all admit Hamiltonian diffeomorphisms with finitely many
periodic orbits.

A viable alternative to the Conley conjecture for such manifolds is
the conjecture that a Hamiltonian diffeomorphism with more fixed
points than necessarily required by the (weak) Arnold conjecture has
infinitely many periodic orbits. (It is possible that in this
conjecture one might need to impose some kind of non-degeneracy
condition (e.g., homological non-triviality) on the fixed points, as
is the case for the version considered in this paper.) For $\CP^n$,
the expected threshold is $n+1$.  This conjecture is inspired by a
celebrated theorem of Franks stating that a Hamiltonian diffeomorphism
(or, even, an area preserving homeomorphism) of $S^2$ with at least
three fixed points must have infinitely many periodic orbits,
\cite{Fr1,Fr2}; see also \cite{FrHa,LeC} for further refinements and
\cite{BH,CKRTZ,Ke:JMD} for symplectic topological proofs.  We will
refer to this analogue of the Conley conjecture as the
\emph{HZ-conjecture}, for, to the best of our knowledge, the first
written account of the assertion is in \cite[p.\ 263]{HZ}.

We find it useful to view the HZ-conjecture in a broader
context. Namely, it appears that the presence of a fixed point that is
unnecessary from a homological or geometrical perspective is already
sufficient to force the existence of infinitely many periodic orbits.
For instance, a theorem from \cite{GG:hyperbolic} asserts that, for a
certain class of closed monotone symplectic manifolds including
$\CP^n$, any Hamiltonian diffeomorphism with a hyperbolic fixed point
must necessarily have infinitely many periodic orbits. (Note that the
original HZ-conjecture, at least for non-degenerate Hamiltonian
diffeomorphisms of $\CP^n$, would follow if one could replace a
hyperbolic fixed point with a non-elliptic one in this theorem.)
Furthermore, there are obvious analogues of the HZ-conjecture for
symplectomorphisms or non-contractible periodic orbits of Hamiltonian
diffeomorphisms. These analogues are also of interest and in some
instances more accessible than the original HZ-conjecture; see, e.g.,
\cite {Ba,GG09generic,GG:hyperbolic,Gu:nc}.

The generalized HZ-conjecture is also the central theme of this paper,
although here we focus on a different aspect of the problem. Our main
result, Theorem \ref{thm:main2}, can be viewed as a ``local version''
of this conjecture, and it holds in all dimensions.  Namely, we prove
a variant of the HZ-conjecture for Hamiltonians on $\R^{2n}$ which are
compactly supported perturbations of certain quadratic forms.  Working
with $\R^{2n}$ allows us to circumvent a number of symplectic
topological obstacles to proving the HZ-conjecture and concentrate on
what we interpret as the dynamical part of the problem, which is still
quite non-trivial.  This is a key difference, technical and
conceptual, between the present work and the approach taken in
\cite{GG:hyperbolic} where the symplectic topology of the ambient
manifold plays a central role.  We use Floer theoretical techniques in
the proofs. Deferring a more detailed discussion of our method to
Section \ref{sec:results}, we merely mention at this point that, for
technical reasons, the quadratic form needs to be hyperbolic. Finally,
it should also be noted that Hamiltonian systems on $\R^{2n}$ with a
controlled (e.g., asymptotically linear) behavior at infinity have
been extensively studied in the context of Hamiltonian mechanics by
classical variational methods; see, e.g., \cite{Ab-book,AZ,An,Cor,
  MW,Ra,ZL,Zou} and references therein. However, to the best of our
knowledge, there is no overlap between that approach and the present
work, including the results.

\subsection{Main results}
\label{sec:results}
To state the main results of the paper, recall that the \emph{mean
  index} $\Delta_H(x)\in\R$ of a periodic orbit $x$ of the Hamiltonian
flow of $H$ measures, roughly speaking, the total angle swept out by
certain eigenvalues with absolute value one of the linearized flow
$d\varphi^t_H$ along $x$; see \cite{Lo,SZ} and also \cite[Section
3.3]{EP} and references therein for a more detailed discussion.  For
instance, the mean index is zero when $d\varphi^t_H$ has no
eigenvalues on the unit circle for any $t\neq 0$, and hence the orbit
is hyperbolic.  Finally, denote by $\Fix(\varphi_H)$ the collection of
fixed points of $\varphi_H$.

\begin{Theorem}
\label{thm:main2}
Let $H \colon S^1 \times \R^{2n} \to \R$ be a Hamiltonian which is
equal to a hyperbolic quadratic form $Q$ at infinity (i.e., outside a
compact set) such that $Q$ has only real eigenvalues. Assume that
$\varphi_H$ has a non-degenerate fixed point with non-zero mean index
and $\Fix(\varphi_H)$ is finite.  Then $\varphi_H$ has simple, i.e.,
uniterated, periodic orbits of arbitrarily large period.
\end{Theorem}

As a consequence, $\varphi_H$ has infinitely many simple periodic
orbits regardless of whether $\Fix(\varphi_H)$ is finite or not.  In
fact, the non-degeneracy condition in Theorem \ref{thm:main2} can be
relaxed and replaced by a much weaker, albeit more technical,
condition that the point is isolated and homologically non-trivial,
i.e., its local Floer homology is non-zero.  This is Theorem
\ref{thm:main}.

\begin{Remark}  
This theorem and Theorem \ref{thm:lowdim} below as well as their
generalizations discussed in Section \ref{sec:proofs}, also hold when
the quadratic form $Q$ has complex eigenvalues $\sigma$, provided that
$|\text{Re } \sigma | > | \text{Im } \sigma | $; see Remark
\ref{rk:complex}.
\end{Remark}

\begin{Remark}
\label{rk:master}
Viewing Theorem \ref{thm:main2} from the perspective of the
generalized HZ-conjecture, observe that the non-degenerate (or
homologically non-trivial) fixed point with non-zero mean index is the
``unnecessary'' point. Furthermore, the presence of one such point $x$
implies the existence of at least two other (homologically
non-trivial) orbits. Indeed, the Floer homology for all iterations of
$H$ is concentrated in degree zero (see Section \ref{sec:Floer}), and
once $k$ is so large that the index of the iterated orbit $x^k$ is
outside the range $[-n,\,n]$, another orbit must take over generating
the homology. Furthermore, there should be at least one more periodic
orbit to cancel out the contribution of $x^k$ to the homology in
higher degrees.
\end{Remark}

Hypothetically, results similar to Theorem \ref{thm:main2} and other
theorems discussed in this section hold when a hyperbolic quadratic
form is replaced by any (autonomous) quadratic form $Q$ without
non-trivial periodic orbits. For instance, in this case, one can
expect to have infinitely many periodic orbits whenever $\varphi_H$
has a non-degenerate fixed point with mean index different from
$\Delta_Q(0)$ or has at least two non-degenerate fixed points; cf.\
Remark \ref{rk:elliptic}. (The latter conjecture, which was the
starting point of this work, is due to Alberto Abbondandolo.)

As has been pointed out above, the proof of Theorem \ref{thm:main2} is
based on Floer theory.  However, for a general quadratic form $Q$,
even when the Floer homology exists, continuation maps fail to have
the desired properties and the homology is not invariant under
iterations. This is the case, for instance, for positive or negative
definite $Q$ (see Remark \ref{rk:posdef}) and the main reason why we
restrict our attention to hyperbolic quadratic forms. Even for such
forms some foundational aspects of Floer theory have to be
reexamined. We do this in Section \ref{sec:maxp}, using, as one could
expect, a version of the maximum principle.

The condition that the fixed point is non-degenerate (and that it has
non-zero mean index) is essential in Theorem \ref{thm:main2}.  For
instance, starting with the flow of $Q(x,y)=xy$ on $\R^2$, it is easy
to introduce degenerate (homologically trivial) fixed points by
slightly perturbing the flow away from the saddle. This way one can
create an arbitrarily large number of fixed points without generating
infinitely many periodic orbits. In fact, we expect some form of
non-degeneracy (e.g., homological non-triviality) to be essential in
the HZ-conjecture beyond the case of $S^2$.

In low dimensions, Theorem \ref{thm:main2} combined with simple index
analysis implies the HZ-conjecture in its original form for
Hamiltonians in question.  To state the result, recall first that
$\varphi_H$ is said to be strongly non-degenerate if all iterations of
$\varphi_H$ are non-degenerate.

\begin{Theorem}
\label{thm:lowdim}
Let $H \colon S^1 \times \R^{2n} \to \R$, with $2n=2$ or $4$, be a
Hamiltonian which is equal to a hyperbolic quadratic form $Q$ at
infinity such that $Q$ has only real eigenvalues. Assume that
$\varphi_H$ is strongly non-degenerate and has at least two fixed
points, and $\Fix(\varphi_H)$ is finite. Then $\varphi_H$ has simple
periodic orbits of arbitrarily large period.
\end{Theorem}

Note that strong non-degeneracy is a $C^{\infty}$-generic condition in
the class of Hamiltonians under consideration.  Let us also point out
that, in contrast with many closed manifolds (see, e.g.,
\cite{GG09generic}), the existence of infinitely many periodic orbits
is obviously not a $C^{\infty}$- or even $C^2$-generic property of
Hamiltonians in Theorem \ref{thm:lowdim}: one has to have an extra
periodic orbit which serves as a seed eventually ``spawning an
infinitude of off-springs".

In dimension two, the strong non-degeneracy requirement can be
relaxed. It suffices to just assume that $\varphi_H$ has at least two
isolated homologically non-trivial fixed points; see Theorem
\ref{thm:R2}. (Also, note that in this case the eigenvalues of a
hyperbolic quadratic form are automatically real.)  However, in
dimension four, non-degeneracy enters the proof in a crucial way.
Finally, note that the two-dimensional case of Theorem \ref{thm:lowdim}
is intimately related to the Franks' theorem; see Remarks
\ref{rk:elliptic} and \ref{rk:Franks}.

 \begin{Remark}
   A more general version of Theorem \ref{thm:lowdim}
   for $2n=2$ was proved in \cite[Theorem 5.1.9]{Ab-book}.
 \end{Remark}

\subsection{Organization of the paper}
In Section \ref{sec:conventions}, we set conventions and notation, and
briefly recall some of the tools used in the paper and provide
relevant references. We establish a version of the maximum principle
and show that the Floer homology as well as the relevant continuation
maps are defined for the class of Hamiltonians in question in Section
\ref{sec:maxp}.  Finally, in Section \ref{sec:proofs}, we prove
Theorems \ref{thm:main2} and \ref{thm:lowdim}.

\subsection{Acknowledgements} The author is grateful to Alberto
Abbondandolo, Viktor Ginzburg, Leonid Polterovich and Cem
Yal\c{c}{\i}n Y{\i}ld{\i}r{\i}m for useful discussions, comments and
suggestions and to the referee for valuable remarks.

\section{Conventions and notation}
\label{sec:conventions}
Throughout the paper, we will be working with the symplectic manifold
$(\R^{2n},\omega)$, where $\omega$ is the standard symplectic form.
All Hamiltonians $H$ considered here are assumed to be one-periodic in
time, i.e., $H\colon S^1 \times \R^{2n} \to\R$, and we set $H_t =
H(t, \cdot)$ for $t\in S^1=\R/\Z$.  The Hamiltonian vector field $X_H$
of $H$ is defined by $i_{X_H}\omega=-dH$. The (time-dependent) flow of
$X_H$ is denoted by $\varphi_H^t$ and its time-one map by
$\varphi_H$. Such time-one maps are referred to as \emph{Hamiltonian
  diffeomorphisms}.  The action of a one-periodic Hamiltonian $H$ on a
loop $\gamma\colon S^1\to \R^{2n}$ is defined by
$$
\CA_H(\gamma)=-\int_z\omega+\int_{S^1} H_t(\gamma(t))\,dt,
$$
where $z\colon D^2\to M$ is such that $z\mid_{S^1}=\gamma$.  The least
action principle asserts that the critical points of $\CA_H$ on the
space of all smooth maps $\gamma\colon S^1\to \R^{2n}$ are exactly the
one-periodic orbits of $\varphi_H^t$.

Let $K$ and $H$ be two one-periodic Hamiltonians. The ``composition''
$K\nat H$ is defined by the formula
\begin{equation}
\label{eq:sum}
(K\nat H)_t=K_t+H_t\circ(\varphi^t_K)^{-1}
\end{equation}
and the flow of $K\nat H$ is $\varphi^t_K\circ \varphi^t_H$.  We set
$H^{\nat k}=H\nat\ldots \nat H$ ($k$ times).  Abusing terminology, we
will refer to $H^{\nat k}$ as the $k$th iteration of $H$. Clearly,
$H^{\nat k}=kH$ when $H$ is autonomous.  (Note that the flow
$\varphi^t_{H^{\nat k}}=(\varphi_H^t)^k$, $t\in [0,\,1]$, is homotopic
with fixed end-points to the flow $\varphi^t_H$, $t\in [0,\,
k]$. Also, in general, $H^{\nat k}$ is not one-periodic, even when $H$
is.)  Furthermore, setting
\begin{equation}
\label{eq:norm}
\| F\|_B=\int_{S^1}\sup_B |F|\, dt
\end{equation}
for a bounded set $B\subset\R^{2n}$, we have $\|H^{\nat k}\|_B=k\|
H\|_B$ when $H$ is autonomous. Note that $\|F\|_B$ is a variant of the
Hofer norm.  (When $F$ is compactly supported on $\R^{2n}$, we will
also use the notation $\| F\|_{\R^{2n}}$ with the obvious meaning.)

The $k$th iteration of a one-periodic orbit $\gamma$ of $H$ will be
denoted by $\gamma^k$. More specifically,
$\gamma^k(t)=\varphi_{H^{\nat k}}^t \left(\gamma(0)\right)$, where
$t\in [0,\,1]$.  We can think of $\gamma^k$ as the $k$-periodic orbit
$\gamma(t)$, $t\in [0,\,k]$, of $H$.  Hence, there is an
action-preserving one-to-one correspondence between one-periodic
orbits of $H^{\nat k}$ and $k$-periodic orbits of $H$.

The \emph{action spectrum} $\CS(H)$ of $H$ is the set of critical
values of $\CA_H$. This is a zero measure, closed (hence nowhere
dense) set; see, e.g., \cite{HZ}. Clearly, the action functional is
homogeneous with respect to iteration:
\begin{equation*}
\label{eq:action-hom}
\CA_{H^{\nat k}}(\gamma^k)=k\CA_H(\gamma).
\end{equation*}

A periodic orbit $\gamma$ of $H$ is said to be \emph{non-degenerate}
if the linearized return map $d\varphi_H \colon T_{\gamma (0)}M\to
T_{\gamma (0)}M$ has no eigenvalues equal to one.  A Hamiltonian is
called non-degenerate if all its one-periodic orbits are
non-degenerate and strongly non-degenerate if all $k$-periodic orbits
(for all $k$) are non-degenerate.

Let $\gamma$ be a non-degenerate periodic orbit. The
\emph{Conley--Zehnder index} $\MUCZ(H,\gamma)\in\Z$ is defined, up to
a sign, as in \cite{Sa,SZ}. (When $H$ is clear from the context we use
the notation $\MUCZ(\gamma)$.)  More specifically, in this paper, the
Conley--Zehnder index is the negative of that in \cite{Sa}.  In other
words, we normalize $\MUCZ$ so that $\MUCZ(\gamma)=n$ when $\gamma$ is
a non-degenerate maximum of an autonomous Hamiltonian with small
Hessian. Furthermore, recall that the mean index $\Delta_H(\gamma)$ is
defined regardless of whether $\gamma$ is degenerate or not, and
$\Delta_H(\gamma)$ depends continuously on $H$ and $\gamma$ in the
obvious sense.  When $\gamma$ is non-degenerate, we have
\begin{equation*}
\label{eq:Delta-MUCZ}
0\leq|\Delta_H(\gamma)-\MUCZ(H,\gamma)|<n.
\end{equation*}
Furthermore, the mean index is also homogeneous with respect to iteration:
\begin{equation*}
\label{eq:index-hom}
\Delta_{H^{\nat k}}(\gamma ^k)=k \Delta_H(\gamma).
\end{equation*}

\section{Maximum principle and Floer homology}
\label{sec:maxp}
Our goal in this section is to show that the Floer homology is defined
and has the standard properties for the class of Hamiltonians in
question. In our setting, essentially the only issue to deal with is
the compactness of moduli spaces of Floer trajectories, which we
establish by proving a version of the maximum principle.

\subsection{Floer homology}
\label{sec:Floer}
Let $Q$ be a hyperbolic quadratic form on $\R^{2n}$, i.e., $Q$ is
non-degenerate and has no eigenvalues on $\text{\bf i} \R$. (Recall
that throughout the paper by eigenvalues of $Q$ we mean the
eigenvalues of the linear Hamiltonian vector field $X_Q$.)  Assume
further that all eigenvalues of $Q$ are real.  (See Remark
\ref{rk:complex} for a variant of the maximum principle when $Q$ has
complex eigenvalues.) Denote by $\hq$ the set of one-periodic
Hamiltonians $H \colon S^1 \times \R^{2n} \to \R$ which are compactly
supported time-dependent perturbations of $Q$. Let $J=J_t$ be a
time-dependent almost complex structure compatible with $\omega$.  We
are interested in solutions $u \colon \R\times S^1 \to \R^{2n}$
of the Floer equation
\begin{equation}
\label{eq:Fl}
\p_s u+J(u) \p_t u =-\nabla H_t (u),
\end{equation}
where $u=u(s,\,t)$ with coordinates $(s,t)$ on $\R\times S^1$ and the
gradient is taken with respect to the one-periodic in time metric
$\left<\cdot\,,\cdot\right>=\omega(\cdot\,,J\cdot)$ on $\R^{2n}$.

In this setting we have:
\begin{Theorem}
\label{thm:maxp}
Let $Q$ be a hyperbolic quadratic form on $(\R^{2n}, \omega)$ with
only real eigenvalues. Then there exists a linear complex structure
$\jq$ compatible with $\omega$ such that whenever $J \equiv \jq$ and
$H \equiv Q$ outside an open ball $B$ with respect to the metric
$\left<\cdot\,,\cdot\right>_Q:=\omega(\cdot\,,\jq\cdot)$, any solution
of \eqref{eq:Fl} for the pair $(H,J)$ that is asymptotic to periodic
orbits of $H$ in $B$ is necessarily contained in $B$.
\end{Theorem}

More generally, consider now solutions $u\colon \Omega\to\R^{2n}$ of
\eqref{eq:Fl}, where $\Omega\subset \R\times S^1$ is an open connected
subset.  Theorem \ref{thm:maxp} is an immediate consequence of the
following proposition which we will refer to as the maximum principle.

\begin{Proposition}[Maximum Principle]
\label{prop:maxp}
Let $Q$ be a hyperbolic quadratic form on $(\R^{2n}, \omega)$ with
only real eigenvalues. Then there exists a linear complex structure
$\jq$ compatible with $\omega$ such that for any solution $u$ (with
domain $\Omega\subset\R\times S^1$) of the Floer equation
\eqref{eq:Fl} for $(Q, \jq)$, the function $\rho =\| u \|^2/2$, where
the norm is induced by the metric $\langle
\cdot\,,\cdot\rangle=\omega(\cdot\,,\jq\cdot)$, cannot attain a
maximum at an interior point of $\Omega$ unless $\rho$ is constant.
\end{Proposition}

\begin{proof}[Proof of Proposition \ref{prop:maxp}]  
  Below we first introduce $\jq$ and then prove that $\rho$ is
  subharmonic on $\Omega$, i.e., $ \Delta \rho \geq 0$, where the
  Laplacian is taken with respect to metric $\omega(\cdot\,,\jq
  \cdot)$.

Since we will be changing the basis and the inner product on the
ambient space throughout the proof, it is more convenient to work with
a hyperbolic quadratic form $Q$ on a finite dimensional symplectic
vector space $(V^{2n},\omega)$. Equip $V$ with a symplectic basis
$(\p_p,\p_q)=(\p_{p_1},\cdots,\p_{p_n}, \p_{q_1},\cdots,\p_{q_n})$
such that in the corresponding coordinates $(p,q)$ on $\R^{2n}$, the
quadratic form $Q$ is expressed as
\begin{equation}
\label{eq:A}
Q(p,q)
= \langle A p, q \rangle.
\end{equation}
Here $A$ is a non-degenerate lower triangular $n\times n$ matrix, and
$\langle \cdot, \cdot \rangle$ is the standard inner product on
$\R^n$.  Indeed, since $Q$ is non-degenerate with only real
eigenvalues, it can be expressed in some symplectic basis
$(\p_p,\p_q)$, as the direct sum of the following normal forms
$$
\sigma \sum_{i=1}^m p_i q_i - \sum_{i=1}^{m-1} p_i q_{i+1},
$$
where $\sigma$ ranges over the positive eigenvalues of $Q$ and $m$ is
the multiplicity of $\sigma$; see \cite{Ar74,wi}. We emphasize that
$p$'s and $q$'s are treated here as vectors in $\R^n$ using the bases
$(\p_p)$ and $(\p_q)$, respectively. (It is clear from this formula
that $A$ is indeed lower triangular.) Note that with this choice all
diagonal entries of $A$, i.e., the eigenvalues of $A$, are positive.

Let $A=D+E$ where $D$ is the diagonal part of $A$ and $E$ is the
strictly lower triangular part. By rescaling the basis vectors
$(\p_p,\p_q)$, while still keeping the basis symplectic and keeping
\eqref{eq:A}, we can make $E$ arbitrarily small. (We will specify
shortly how small $E$ has to be. Here we merely note that the
rescaling does not effect $D$ and that, in fact, $E$ is required to be
small compared to $D$.) We keep the notation $(\p_p,\p_q)$ for the new
basis and $(p,q)$ for the resulting linear coordinates.

The complex structure $\jq$ is defined by the requirement $\jq
\p_p=-\p_q$. This structure is compatible with $\omega$, and we denote
by $\left<\cdot,\,\cdot\right>_Q$ the resulting inner product
$\omega(\cdot,\jq \cdot)$ on $V$, i.e.,
$\left<\cdot,\,\cdot\right>_Q:=\omega(\cdot,\jq \cdot)$.  From now on
we identify $(V,\omega)$ with the standard symplectic $\R^{2n}$ using
the basis $(\p_p,\p_q)$. Under this identification, $\jq$ becomes the
standard complex structure on $\R^{2n}=\C^n$, and
$\left<\cdot,\,\cdot\right>_Q$ turns into the standard inner
product. Note also that the restriction of
$\left<\cdot,\,\cdot\right>_Q$ to the subspaces generated by $\p_p$
and $\p_q$, respectively, is the inner product $\langle \cdot, \cdot
\rangle$ on $\R^n$. Finally, we emphasize that all these structures,
except for $\omega$, depend on the choice of the basis $(\p_p,\p_q)$
which is to be finalized below (after we state how small $E$ needs to
be).

In what follows, we calculate the Laplacian with respect to the metric
$\langle \cdot\,,\cdot\rangle_Q$, where we set $u(s,t)=(p,q)$ and use
the Floer equation \eqref{eq:Fl}:
\begin{equation}
\label{eq:Laplacian}
\begin{split}
\Delta \rho & = \rho_{ss} + \rho_{tt} \\
            & = {\lVert u_s \rVert}^2 + {\lVert u_t \rVert}^2 - 
                \langle u, \p_s\nabla Q(u)\rangle_Q + 
                \langle u, \jq \p_t\nabla Q(u) \rangle_Q \\
            & = {\lVert u_s \rVert}^2 + {\lVert u_t \rVert}^2
               + \langle A^2 p, p \rangle 
               + \langle A^2 q, q \rangle \\
            & \quad + \langle p,\left( A-A^T \right)q_s \rangle
                    - \langle q,\left( A-A^T \right)p_s \rangle \\            
            & = {\lVert u_s \rVert}^2 + {\lVert u_t \rVert}^2 
              + {\lVert Dp \rVert}^2 + {\lVert Dq \rVert}^2  
              + \langle E^2 p, p \rangle + \langle E^2 q, q \rangle\\
            & \quad + \langle \left( DE+ED \right)\, p, p \rangle 
                    + \langle \left( DE+ED \right)\, q, q \rangle \\
            & \quad + \langle p,\left( E-E^T \right)q_s \rangle
                    - \langle q,\left( E-E^T \right)p_s \rangle.
\end{split}
\end{equation}

Next, we specify the requirements on $E$. To this end, let $ \lambda
:= \min \lambda_i >0$, where $\lambda_i$'s are the
eigenvalues of $A$ (or $D$). Then we have
\begin{equation}
\label{eq:lambda}
\|Dx\|^2 \geq \lambda^2 \| x \|^2
\text{ for any }x \in \R^n.
\end{equation}
Now, $E$ is required to be so small that (i), (ii) and (iii) below
hold:
\begin{itemize}
\item[(i)]
$
\lvert \langle E^2 x, x \rangle \rvert \le 
\dfrac{\lambda^2}{10} {\lVert x \rVert}^2 
\text{ for any }x \in \R^n 
$,
\vspace{.1in}
\item[(ii)] 
$
\lvert \langle DE \,x, x \rangle \rvert \le
\dfrac{\lambda^2}{20}{\lVert x \rVert}^2
\text{ and }
\lvert \langle ED \,x, x \rangle \rvert \le
\dfrac{\lambda^2}{20}{\lVert x \rVert}^2
\text{ for any }x \in \R^n 
$,
\vspace{.1in}
\item[(iii)] 
$
\lvert \langle x , \left( E-E^T \right)y \rangle \rvert \le
\dfrac{\lambda}{8}{\lVert x \rVert} {\lVert y \rVert} 
\text{ for any }x \text{ and }y \in \R^n 
$.
\vspace{.1in}
\end{itemize}
Using \eqref{eq:lambda}, and (i) and (ii) for $x=p$ and $x=q$, and
(iii) for $(x,\,y)=(p,\,q_s)$ and $(x,\,y)=(q,\,p_s)$ in
\eqref{eq:Laplacian}, it is straightforward to show that
\begin{equation}
\label{eq:Laplacian2}
\begin{split}
\Delta \rho & \geq \dfrac{3 \lambda^2}{10} {\lVert u \rVert}^2
            + {\lVert u_s \rVert}^2 
            + \dfrac{\lambda^2}{2} {\lVert u \rVert}^2 
            - \dfrac{\lambda}{4}{\lVert u \rVert}{\lVert u_s \rVert}\\
            & \geq \dfrac{3 \lambda^2}{10} {\lVert u \rVert}^2 +
             {\left(\lVert u_s \rVert - 
             \dfrac{\lambda}{\sqrt{2}}{\lVert u \rVert} \right)}^2\\
            & \geq \dfrac{3 \lambda^2}{10} {\lVert u \rVert}^2 \geq 0. 
\end{split}
\end{equation}
\end{proof}

\begin{Remark}
\label{rk:complex}
It is not hard to see that Proposition \ref{prop:maxp} still holds
when the quadratic form $Q$ has complex eigenvalues $\sigma$, provided
that $\lvert \text{Re } \sigma\rvert > \lvert \text{Im } \sigma \rvert
$ or, equivalently, $\text{Re } \sigma^2>0$ for all
eigenvalues. However, in general, without this assumption (or when $Q$
is elliptic, but not positive definite), there seems to be no reason
to expect the maximum principle to hold. There are also several other
variants of the maximum principle which hold for solutions of the
Floer equation for $Q$. For instance, it holds for the functions $\|
p\|^2$ and $\| q\|^2$ separately.
\end{Remark}

As a consequence of Theorem \ref{thm:maxp}, the total and filtered
Floer homology groups of $H\in \hq$, denoted by $\HF(H)$ and
$\HF^{(a,\,b)}(H)$, respectively, are defined and have properties
similar to those for closed symplectically aspherical manifolds; see,
e.g, \cite{HZ,MS}.  (For the sake of simplicity all homology groups
are taken over $\Z_2$.)  Likewise, the local Floer homology $\HF
(H,\gamma)$ of $H$ at an isolated periodic orbit $\gamma$ is also
defined and has the usual properties; see, e.g.,
\cite{F:89witten,Fl89fp,Gi:conley,GG:gap}.  (Here $J$ is an
$\omega$-compatible almost complex structure which is generic within
the class of almost complex structures equal outside a compact set to
some $\jq$ as in Theorem \ref{thm:maxp}. We will discuss the
dependence of the Floer homology on $J$ shortly.)

Our next goal is to define the continuation maps induced by homotopies
of Hamiltonians in $\hq$. To this end, we say that a homotopy
$F_s=Q+f_s$ in $\hq$ from $H_0$ to $H_1$ is compactly supported if
$\bigcup_s \supp f_s$ is bounded. (Observe that this is not
automatically the case.) Then we have a continuation map
\begin{equation}
\label{eq:homotopy}
\Psi\colon \HF^{(a,\,b)}(H_0) \to \HF^{(a,\,b)+C}(H_1)
\end{equation}
for any $H_0$ and $H_1$ in $\hq$, induced by a homotopy $F_s$ in
$\hq$. Here $(a,\,b)+C$ stands for $(a+C,\,b+C)$ and
\begin{equation}
\label{eq:C}
C \geq \int_{-\infty}^{\infty} \int_{S^1} 
\sup_{\R^{2n}} \partial_s F_s \,dt\,ds
=
\int_{-\infty}^{\infty} \int_{S^1} 
\sup_{\R^{2n}} \partial_s f_s \,dt\,ds
\end{equation}
with $\partial_s F_s \equiv 0$ when $|s|$ is large; see \cite[Section
3.2.2]{Gi:coiso}. Note that the supremum in \eqref{eq:C} exists since
the homotopy is compactly supported.

We now have $\HF(H) = \HF(Q) \cong \Z_2 $, concentrated in degree
$\MUCZ(Q,0)=0$. It is clear that the filtered Floer homology
$\HF^{(a,\,b)}(H)$ is independent of the almost complex structure $J$
as long as $\jq$ is fixed. However, it is not obvious at all whether
this homology is independent of the choice of $\jq$. In what follows,
we will always have $\jq$ fixed and suppress this hypothetical
dependence in the notation.

\subsection{Continuation maps beyond $\hq$}
\label{sec:homotopy}
The class $\hq$ is not closed under iterations. For instance, $H^
{\nat 2} \in \CH_{\scriptscriptstyle{2Q}}$ when $H \in \hq$. To
incorporate iterations into the picture, we consider a broader class
$\hhq$ which is the union of the classes $
\CH_{\scriptscriptstyle{kQ}} $ for all real $k>0$. Clearly, this class
is now closed under iterations. Moreover, one can see from the proof
of Proposition \ref{prop:maxp} that there exists a common almost
complex structure, $\jq$, for which Theorem \ref{thm:maxp} holds for
all Hamiltonians in $\hhq$ or, to be more precise, any $\jq$ can also
be taken as $\jkq$.  (The reason is that $\jq$ is determined by the
requirement that the off-diagonal part $E$ of $A$ is small
\emph{compared to} $D$, rather than just small. Thus, if conditions
(i), (ii) and (iii) are satisfied for $Q$, they are also automatically
satisfied in the same basis for $kQ$ for any $k>0$.)  From now on we
fix $\jq$.

As above, we say that a homotopy $F_s=k(s)Q+f_s$ in $\hhq$ is
compactly supported if $\bigcup_s \supp f_s$ is bounded and call the
closure of this union the support of the homotopy.  A homotopy is
called \emph{slow} if it is compactly supported and, say,
\begin{equation}
\label{eq:slow}
\frac{|k'(s)|}{(k(s))^2}\leq \frac{3\lambda^2}{20}
\inf_{x\in \R^{2n} \setminus \{0\} }\frac{\|x\|^2}{|Q(x)|},
\end{equation}
where $\lambda=\min \lambda_i $ is as in Section \ref{sec:Floer}; cf.\
\cite{cfh}. Clearly, the right hand side in \eqref{eq:slow} is
positive. (In fact, the infimum in \eqref{eq:slow} is equal to
$1/\lambda_{\max}$ where $\lambda_{\max}$ is the largest of the
absolute values of the eigenvalues of $Q$ with respect to
$\|x\|^2$. This follows from the Courant--Fischer minimax theorem;
see, e.g., \cite[Chapter 5]{De}.) Recall also that a homotopy $F_s$
from $H_0$ to $H_1$ is called linear if
$F_s=\left(1-g(s)\right)H_0+g(s)H_1$, where $g \colon \R \to \R$ is an
increasing smooth function equal to zero for $s \ll 0$ and one for
$s\gg 1$.

\begin{Theorem}
\label{thm:homotopy}
Let $F_s$ be a slow homotopy in $\hhq$ from $H_0$ to $H_1$, supported
in a ball $B$ with respect to the metric
$\left<\cdot\,,\cdot\right>_Q=\omega(\cdot\,,\jq\cdot)$. The
continuation map $\Psi$ as in \eqref{eq:homotopy} is defined, where
$C$ satisfies
\begin{equation}
\label{eq:C2}
C \geq \int_{-\infty}^{\infty} \int_{S^1} 
\sup_{B} \partial_s F_s \,dt\,ds
\end{equation}
This map is independent of the slow homotopy. Furthermore, for a
linear slow homotopy, we can take 
$C = \| H_1 - H_0 \|_B 
:= \int_{S^1}\sup_B |H_1 - H_0 |\, dt
$.
\end{Theorem}
 
It is not hard to see that any compactly supported homotopy in $\hhq$
can be reparametrized to make it slow without changing the right hand
side in \eqref{eq:C2}. Note also that, although the notion of a slow
homotopy is independent of the size of the support, the lower bound in
\eqref{eq:C2} does depend in general on the ball $B$ containing the
support and increases with the size of $B$. However, one can show that
the continuation map $\Psi$ is independent of the ball $B$ in the
following sense: whenever for a fixed homotopy and two different balls
$C$ satisfies \eqref{eq:C2} for both of the balls, the resulting
continuation map is independent of the ball. (In what follows we will
not use this fact.)

The continuation maps $\Psi$ have properties similar to their
counterparts in the ordinary Floer homology. For instance, the
continuation map induced by a concatenation of homotopies is equal to
the composition of the continuation maps, and continuation maps
commute with the maps in the long exact sequence in filtered Floer
homology. (See \cite{Gi:coiso} for a detailed account on the so-called
$C$-bounded homotopies in filtered Floer homology.) Note, however,
that here, as in \eqref{eq:homotopy}, the almost complex structure
$J_s$ is independent of $s$ outside $B$.

\begin{Remark}
\label{rem:shift}
We emphasize that the continuation map $\Psi$ is not necessarily
defined when the homotopy is not slow.
\end{Remark}

\begin {proof}[Proof of Theorem \ref{thm:homotopy}]
  To prove that $\Psi$ is well-defined, it suffices to show that the
  maximum principle, Proposition \ref{prop:maxp}, extends to solutions
  of the Floer equation for slow homotopies $Q_s=k(s)Q$ connecting
  $Q_0=k_0Q$ and $Q_1=k_1Q$, where $k_0=k(0)$ and $k_1=k(1)$. To this
  end, note that $Q_s = k(s) \langle A p, q \rangle$ and recall that,
  as was noted above, we can take $\jkq$ to be $\jq$. 
  Calculating the Laplacian with respect to the metric
  $\left<\cdot\,,\cdot\right>_Q=\omega(\cdot\,,\jq\cdot)$
 in this setting, we obtain
\begin{equation*}
\begin{split}
\Delta \rho & = {\lVert u_s \rVert}^2 + {\lVert u_t \rVert}^2
              + k(s)\left(\langle p,\left( A-A^T \right)q_s \rangle
              - \langle q,\left( A-A^T \right)p_s \rangle \right)\\
            & \quad + (k(s))^2 \left( \langle A^2 p, p \rangle 
                    + \langle A^2 q, q \rangle \right) 
                    - 2 k'(s) \langle Ap, q\rangle \\  
            & = {\lVert u_s \rVert}^2 + {\lVert u_t \rVert}^2 
              + (k(s))^2 \left( {\lVert Dp \rVert}^2 
              + {\lVert Dq \rVert}^2 
              + \langle E^2 p, p \rangle 
              + \langle E^2 q, q \rangle \right)\\
            & \quad + (k(s))^2 \left(
                      \langle \left( DE+ED \right)\, p, p \rangle 
                    + \langle \left( DE+ED \right)\, q, q \rangle \right)\\
            & \quad + k(s)\left( \langle p,\left( E-E^T \right)q_s \rangle
                    - \langle q,\left( E-E^T \right)p_s \rangle \right)\\
            & \quad  - 2 k'(s) \langle Ap, q\rangle \\
            & \ge \frac{3 \lambda^2}{10} (k(s))^2{\lVert u \rVert}^2
              + \left(
                \lVert u_s \rVert - \frac{\lambda \lvert k(s) \rvert} 
                {\sqrt{2}} \lVert u \rVert
                \right)^2 
              -2 \lvert k'(s) \rvert \cdot \lvert Q(u) \rvert.
\end{split}
\end{equation*}
Hence $\Delta\rho\geq 0$ by \eqref{eq:slow}.

Now, the facts that we can take $C$ satisfying \eqref{eq:C2} for a
general slow homotopy and that $C = \| H_1- H_0 \|_B$ satisfies
\eqref{eq:C2} for a linear slow homotopy are established by a standard
calculation (see, e.g., \cite{Gi:coiso,sc}) combined with the
observation that homotopy trajectories (i.e., solutions of
\eqref{eq:Fl} for the pair $(F_s, J_s)$) are confined to $B$ due to
the maximum principle.
\end{proof}

\begin{Remark}
\label{rk:posdef}
  The maximum principle is also known to hold for positive definite
  quadratic Hamiltonians; see \cite{McD,vi} and also \cite{Sei}. This
  fact underlies the definition of symplectic homology and, in fact,
  it was the motivation of our approach in this paper. However, it is
  worth pointing out that in this case the continuation map between a
  Hamiltonian equal to $kQ$ at infinity and the one equal to $(k+1)Q$
  at infinity is defined only in one direction and this map, depending
  on $Q$, may be zero. This is the main reason why our approach to the
  proof of Theorem \ref{thm:main2} does not carry over to positive
  definite quadratic Hamiltonians.
\end{Remark}

\section{Proofs and generalizations}
\label{sec:proofs}

\subsection{Proof of Theorem \ref{thm:main2}}
\label{sec:proof-main}
As has been mentioned in the introduction, we establish a more general
result. To state it, recall again that an isolated periodic orbit $x$
is said to be homologically non-trivial if the local Floer homology of
$H$ at $x$ is non-zero. For instance, a non-degenerate fixed point is
homologically non-trivial. More generally, an isolated fixed point
with non-vanishing topological index is homologically non-trivial; for
this index is equal, up to a sign, to the Euler characteristic of the
local Floer homology.  The notion of homological non-triviality seems
to be particularly well-suited for use in the context of HZ- and
Conley conjectures; see, e.g., Remark \ref{rk:master}. Theorem
\ref{thm:main2} is an immediate consequence of the following result.

\begin{Theorem}
\label{thm:main}
Let $H \colon S^1 \times \R^{2n} \to \R$ be a Hamiltonian which is
equal to a hyperbolic quadratic form $Q$ at infinity (i.e., outside a
compact set) such that $Q$ has only real eigenvalues. Assume that
$\varphi_H$ has an isolated homologically non-trivial fixed point $x$
with non-zero mean index and $\Fix(\varphi_H)$ is finite.  Then
$\varphi_H$ has simple periodic orbits of arbitrarily large period.
\end{Theorem}

\begin{proof}[Proof of Theorem \ref{thm:main}]
In what follows, for the sake of brevity, we suppress the
$t$-dependence when taking a supremum or specifying the support of a
function. For instance, when we say that a function is supported in $Y
\subset \R ^{2n}$, we mean that the support is in $ S^1 \times Y$.
Likewise, two functions are equal on $Y$ means that they are equal on
$S^1 \times Y$, etc. Finally, the supremum, without a set specified,
will stand for the supremum over $\R^{2n}$.

Let $H=Q+f$ as in the statement of the theorem. Pick a polyball
$P=B^n\times B^n$ containing $\supp f$ and a ball $V \supset P$.
Throughout the proof, as in Section \ref{sec:maxp}, we assume that the
off-diagonal part $E$ of $A$ is small enough when compared to the
diagonal part $D$. In particular, every integral curve of the flow of
$Q$ intersect $P$ along a connected set.  Before we actually turn to
the proof of the theorem, we need to first modify $H$, without
essentially changing its dynamics, to control the energy shift
resulting from the homotopy between different iterations of~$H$.

\begin{Lemma}
\label{lemma:Qtilde}
There exist constants $C_1 >0$ and $C_2>0$, depending only on the
quadratic form $Q$ and the ball $V$, such that for every $\eps \in
(0,1]$ there exists an autonomous Hamiltonian $\tQ$ with the following
properties:
\begin{itemize}
\item[(i)] $\tQ=Q$ on $V$,
\item[(ii)] $\tQ = \eps Q$ outside a ball $V_\eps \supset V$ of radius
  $R=C_1/ \sqrt{\eps}$,
\item[(iii)] $\displaystyle \sup_{V_\eps} \lvert \tQ \rvert = C_2$,
\item[(iv)] The Hamiltonian flow of $\tQ$ has no periodic orbits other
  than the origin, and every integral curve of its flow intersects $P$
  along a connected set.
\end{itemize}   
\end{Lemma}
The essential point here is that the constants $C_1$ and $C_2$ are
independent of $\eps$ while $R=C_1/\sqrt{\eps}$ (but not, say, of
order $1/\eps$). We will prove this lemma by giving an explicit
construction of $\tQ$ after the proof of Theorem \ref{thm:main}. One
can think of $\tQ$ as a family of Hamiltonians smoothly parametrized
by $\eps$ with $\tQ=Q$ for $\eps=1$.

Consider now the Hamiltonian
\begin{equation*}
\label{eq:Htilde}
\tH = \tQ + f = \eps Q + ( \tQ - \eps Q ) + f = \eps Q + h,
\end{equation*}
where $h = (\tQ - \eps Q ) + f $ is supported in $V_\eps$.  Observe
that $\tH \in \hhq$. Furthermore, $\tH=H$ in $V$, the ball where the
Hamiltonians have non-trivial dynamics. Moreover, for every period,
$\tH$ and $H$ have exactly the same periodic orbits by Lemma
\ref{lemma:Qtilde} (iv), and the orbits have the same actions and
indices.  In fact, one might expect these Hamiltonians to have exactly
the same filtered Floer homology with isomorphism induced by a slow
linear homotopy. However, we have not been able to prove this fact.

The next lemma concerns the iterations $\tH^{\nat k}$ and an estimate,
independent of $k$, of the difference $\tH^{\nat {(k+\ell})} -
\tH^{\nat k}$, which will be essential for the proof of
Theorem~\ref{thm:main}.

\begin{Lemma}
  \label{lemma:shift} The Hamiltonian $\tH^{\nat k}$ satisfies the
  following conditions:
\begin{itemize}
\item[(i)]$\tH^{\nat k} \in \hhq$ and is equal to $k\eps Q$ outside
  the ball $B_k$ of radius $\| \varphi^{\eps (k-1)}_Q \|R $ centered
  at the origin, where $\varphi^{\eps (k-1)}_{Q}$ is viewed as a
  linear operator.
\item[(ii)] Assume that $k$, $\ell$ and $\eps$ are such that $\|
  \varphi^{\eps(k+\ell-1)}_{Q}\| \leq 2$. Then
\begin{equation}
\label{eq:shift}
\| \tH^{\nat {(k+\ell)}} - \tH^{\nat k} \|_{B_{k+\ell}} \leq C_3 \ell,
\end{equation}
where $C_3$ is independent of $k$, $\ell$ and $\eps$, and the norm is
as defined in \eqref{eq:norm}.
\end{itemize}
\end{Lemma}

\begin{proof}[Proof of Lemma \ref{lemma:shift}]
Denote by $B_1$ the ball $V_\eps$ from Lemma \ref{lemma:Qtilde}, i.e.,
$B_1$ is the ball of radius $R=C_1/\sqrt{\eps}$ centered at the
origin. Consider the nested sets
$$
Y_k= \bigcup_{t \in [0,\,1]} \varphi_{\eps Q}^{(k-1)t}(B_1) 
\text{ for } k \in \N. 
$$ 
Let $B_k=B(R_k)$ be the ball of radius $R_k = \| \varphi^{(k-1)}_{\eps
  Q} \|R$. Clearly, $B_k \supset Y_k$.

Recall that $\tH=\eps Q + h$, where $h = (\tQ - \eps Q ) + f $ is
supported in $B_1$. Observe that $\tH^{\nat k}$ can be expressed as
\begin{equation*}
\label{eq:H-iteration}
\tH^{\nat k} = 
k \eps Q + \sum_{j=0}^{k-1} h \circ (\varphi_{\tH}^t)^{-j} +
\eps \sum_{j=0}^{k-1} \left(Q \circ (\varphi_{\tH}^t)^{-j} - Q \right)
= k \eps Q + h_k,
\end{equation*}
We now show that $\supp h_k \subset Y_k$, which settles (i). Since
$\supp h \subset B_1$, a point $x$ can be in $\supp (h \circ
(\varphi_{\tH}^t)^{-j} )$ only if $ (\varphi_{\eps Q}^{\tau})^{-j}(x)
\in B_1 $ for some $\tau \in [0,\,t]$.  This implies that
$$
x \in \left( (\varphi_{\eps Q}^{\tau} )^{-j} \right)^{-1} (B_1) =
\varphi_{\eps Q}^{j \tau} (B_1) \subset \bigcup_{t \in [0,\,1]}
\varphi_{\eps Q}^{jt}(B_1) = Y_j.
$$
Hence, the
first term in $h_k$ is supported in $Y_k$.  Dealing with the second
term in $h_k$, we first note that
$$
\eps \big( Q \circ (\varphi_{\tH}^t)^{-j} - Q \big)
= \eps \big( Q \circ (\varphi_{\tH}^t)^{-j} 
- Q \circ (\varphi_{\eps Q}^t)^{-j} \big)
$$
since $Q$ is autonomous. Now, it is clear that $
(\varphi_{\tH}^t)^{-j}(x) \neq (\varphi_{\eps Q}^t)^{-j} (x) $ only
when the integral curve of $\eps Q$ through $x$ for $[-jt,\, 0]$
enters $B_1$, i.e., $ (\varphi_{\eps Q}^{\tau})^{-j}(x) \in B_1$ for
some $\tau \in [0,\,t]$. Hence, similarly to the first term, the
second term in $h_k$ is also supported in $Y_k$, and we have $\supp
h_k \subset Y_k$.

To establish (ii), denote by $B(1)$ the unit ball and observe that for
any $j= k, \cdots, k+\ell-1$, we have
\begin{equation*}
\begin{split}
\sup_{B_{k+\ell}} \big| h \circ (\varphi_{\tH}^t)^{-j} \big|
&=
\sup_{B_1} | \tQ - \eps Q + f | \\
&\leq 
\sup_{B_1} |\tQ| + \eps \sup_{B_1} |Q| + \sup |f| \\
&\leq 
C_2 + \eps \frac{C_1^2}{\eps} \sup_{B(1)} |Q| +  \sup |f|\\
&=C_2 + C_1^2 \sup_{B(1)} |Q| +  \sup |f|.
\end{split}
\end{equation*}
Furthermore, by the energy conservation law, we have
\begin{equation*}
\begin{split}
\sup_{B_{k+\ell}} \big| Q \circ (\varphi_{\tH}^t)^{-j} - Q \big|
&=
\sup_{B_{k+\ell}} \big| Q \circ (\varphi_{\tH}^t)^{-j} 
                     - Q \circ (\varphi_{\eps Q}^t)^{-j} \big|\\
&=
\sup_{B_1} \big| Q \circ (\varphi_{\tH}^t)^{-j} 
              - Q \circ (\varphi_{\eps Q}^t)^{-j} \big|\\
& \leq 2 \sup_{B_1} |Q| \\
& \leq 2 \frac{C_1^2}{\eps} \sup_{B(1)} |Q|.
\end{split}
\end{equation*}
Now, recall that $k$, $\ell$ and $\eps$ are such that $\|
\varphi^{\eps(k+\ell-1)}_{Q}\| \leq 2$.  Thus $B_{k+\ell} \subset 2
B_1$. Setting $M = \displaystyle \sup_{B(1)} |Q|$ and using the above
estimates, we then have:
\begin{equation*}
\begin{split}
\sup_{B_{k+\ell}} | \tH^{\nat {(k+\ell)}} - \tH^{\nat k} |
  & = \sup_{B_{k+\ell}} \Big| \ell \eps Q +
      \sum_{j=k}^{k+\ell-1} h \circ (\varphi_{\tH}^t)^{-j}
      + \eps \sum_{j=k}^{k+\ell-1} 
      \left( Q \circ (\varphi_{\tH}^t)^{-j} - Q \right) \Big|\\
  & \leq \ell \eps \sup_{2B_1} |Q| 
       + \sum_{j=k}^{k+\ell-1} \sup_{B_{k+\ell}} \Big| h \circ
       (\varphi_{\tH}^t)^{-j}\Big| \\
       & \quad + \eps  \sum_{j=k}^{k+\ell-1} \sup_{B_{k+\ell}} \Big| 
         \left( Q \circ (\varphi_{\tH}^t)^{-j} - Q \right) \Big|\\
  & \leq 4 \ell \eps \frac{C_1^2}{\eps} M
       + \ell \left( C_2
       +  C_1^2 M 
       + \sup|f| \right)
       + 2 \ell \eps \frac{C_1^2}{\eps} M \\
  & \leq 
\big(7 C_1^2 M + C_2 + \sup |f|\big) \ell
=: C_3 \ell
\end{split}
\end{equation*}
where $C_3$ is independent of $k$, $\ell$ and $\eps$. Finally, we have
\begin{equation}
\| \tH^{\nat {(k+\ell)}} - \tH^{\nat k} \|_{B_{k+\ell}} 
=
\int_{S^1} \sup_{B_{k+\ell}} | \tH^{\nat {(k+\ell)}} - \tH^{\nat k} | dt
\leq
C_3\ell,
\end{equation}
which concludes the proof of Lemma \ref{lemma:shift}.
\end{proof}

From now on we will work with the Hamiltonians $\tH$, and at this stage
we prefer not to specify the parameter $\eps$ yet. These Hamiltonians
have the same periodic orbits with the same actions and indices, and
up to the point when the homotopy between the iterated Hamiltonians is
considered the argument applies to any of the Hamiltonians $\tH$.

It is worth mentioning again that the Hamiltonians $\tH^{\nat k}$ are
not one-periodic in time even though $\tH$ is. This issue, however, is
quite standard and can be dealt with in a straightforward way. Namely,
consider a Hamiltonian $G = K + g$, where $g=g_t$ is time-dependent
for $t \in [0,\,1]$ and $K$ is any autonomous Hamiltonian. The
Hamiltonian diffeomorphism $ \varphi_G $ can be generated by a
one-periodic Hamiltonian
$$
\bar{G}= K + \lambda'(t) g_{\scriptscriptstyle{\lambda (t)}} \circ
\varphi_K^{ \scriptscriptstyle{ \lambda(t)-t } },
$$
where $\lambda \colon [0,1] \to [0,1]$ is an increasing function equal
to zero for $t\approx0$ and one for $t\approx 1$. We apply this
procedure to $\tH^{\nat k}$ with $K=k\eps Q$ and $g=h_k$. The actions,
the Conley-Zehnder indices, and the mean indices of the periodic
orbits do not change.  The change of the set $B_k$ can be made
arbitrarily small, of the order $ \|1 -\lambda'(t)\|_{L^1}$.  As a
consequence, the upper bound \eqref{eq:shift} can also be adjusted by
an arbitrarily small amount independent of $k$. In what follows, we
will treat the Hamiltonians $\tH^{\nat k}$ as one-periodic in time,
allowing for these straightforward modifications.

Now we are in a position to proceed with the proof.  It suffices to
show that there exist arbitrarily large primes which occur as periods
of simple periodic orbits.  Arguing by contradiction, assume that only
finitely many prime numbers are attained as the periods.  From now on,
we always denote by $p$ or $p_i$ a prime number greater than the
largest period.  Let $\tH=\tQ+f$ where $\tQ$ is any Hamiltonian from
Lemma \ref{lemma:Qtilde}.  Then for any such prime $p$ all
$p$-periodic orbits of $\varphi_{\tH}$ are iterations of fixed points
of $\varphi_H$, and hence $\CS(\tH^{\nat p}) = p\, \CS(H)$. Recall in
this connection that $\varphi_H$ is assumed to have finitely many
fixed points.  Next, let us note that all sufficiently large prime
numbers are admissible in the sense of \cite{GG:gap}. Thus, under such
iterations of $\tH$, the orbit $x$ stays isolated, and
$$
\HF(\tH^{\nat p}, x^{p}) = \HF(H^{\nat p},x) =\HF(H,x) 
$$
up to, in the second equality, a shift of degree determined by the 
order of iteration $p$; see \cite[Theorem 1.1]{GG:gap}. In particular,
in our case, $\HF(\tH^{\nat p}, x^{p}) \neq 0$ since $\HF(H,x)\neq 0$.

As has been mentioned above, $\Delta_{\tH}(x)=\Delta_H(x)$, and let us
assume that $\Delta_H (x) > 0$, for the argument is similar if
$\Delta_H (x) < 0$.  Moreover, let us assume for the sake of
simplicity that $\CA_H(x)=0$ and hence $\CA_{\tH}(x)=0$. (The general
case can be dealt with in a similar fashion and requires only
notational modifications.) Consequently, $\CA_{{\tH}^{\nat p}}(x^p)=0$
for all iterations $p$. Let $a>0$ be outside $\CS(\tH)=\CS(H)$ such
that $0$ is the only point in $(-a,\,a) \cap \CS(H)$ and therefore in
$(-ap,\,ap) \cap \CS(\tH^{\nat p})$. Then we have
\begin{equation}
\label{eq:locFH}
\HF^{(-ap,\,ap)}_*(\tH^{\nat p})=
\HF_*(H^{\nat p}, x^p) \oplus \hdots,
\end{equation}
where the dots represent the local Floer homology contributions from
the fixed points with zero action other than $x$. Furthermore, we
henceforth focus on degrees $*$ such that $|*|>n$. This guarantees
that the fixed points with zero mean index do not contribute to
$\HF^{(-ap,\,ap)}_*(\tH^{\nat p})$, for their local Floer homology
groups are supported in $[-n,\,n]$, where the support is, by
definition, the set of degrees for which the local Floer homology
groups are non-zero.  Thus all terms on the right hand side of
\eqref{eq:locFH} come from fixed points with non-zero mean
index. Moreover, we can further restrict $*$ so that only the fixed
points $\gamma$ having the same mean index as $x$ contribute to the
right hand side of \eqref{eq:locFH}. This is possible since the
supports of local Floer homology groups coming from fixed points with
other non-zero mean indices are separated from $\supp \HF_*(H ^{\nat
  p}, x^{p}) \subset [p\Delta_{H} (x)-n, \, p\Delta_{H} (x)+n ] $
whenever $p$ is sufficiently large.

From now on, we work with primes $p>2$ which are as large as is needed
above. Let us order these prime numbers as $p_1 < p_2 < \hdots$. In
what follows, $p_i$ always denotes a prime from this sequence.

Next, recall that $\Delta_H (x) > 0$ and let $m \in \N$ be such that
$m>n/ \Delta_H (x)$. Then, using the fact that $p_{i+m}-p_i \geq 2m$,
we see that the supports of $\HF (H^{\nat p_i}, \gamma^{p_i})$ and $
\HF (H^{\nat {p_{i+m}}}, \gamma^{p_{i+m}}) $ are disjoint for all $i$
and for all fixed points $\gamma$ of $\varphi_H$ with
$\CA_H(\gamma)=0$ and $\Delta_H(\gamma)=\Delta_H(x)$. This is because
$$
[p_i\Delta_H (x)-n,\,p_i\Delta_H (x)+n]
\cap
[p_{i+m}\Delta_H(x)-n,\,p_{i+m}\Delta_H (x)+n] 
= \emptyset,
$$
where the first interval contains $\supp \HF(H^{\nat p_i},
\gamma^{p_i})$ and the second one contains $ \supp \HF(H^{\nat
  p_{i+m}}, \gamma^{p_{i+m}})$. Moreover, for any $p_i$, there exists
an integer $s_i$ such that $ \HF_{s_i}(H^{\nat {p_i}}, x^{p_i}) \neq 0
$ as is mentioned earlier and proved in \cite{GG:gap}. Thus we see
that
\begin{equation}
\label{eq:r}
\HF_{s_i}(H^{\nat {p_i}}, x^{p_i}) \neq 0 \textrm{ and } 
\HF_{s_i} (H^{\nat {p_{i+m}}}, \gamma^{p_{i+m}}) = 0
\end{equation}
for all fixed points $\gamma$ as above, since $s_i$ is outside $\supp
\HF(H^{\nat p_{i+m}}, \gamma^{p_{i+m}})$ for all such $\gamma$.

Choose $p_i$ so large that $p_i a > 6 C_3 (p_{i+m}-p_i) $ where $C_3$
is introduced in Lemma \ref{lemma:shift}. The latter is guaranteed for
large primes by the fact that $p_{i+1}-p_i = o(p_i)$; see
\cite{BHP}. (Obviously, one can write $p_{i+m}-p_i$ as a telescoping
sum of the differences of two consecutive primes, and hence, by a
simple inductive argument, $p_{i+m}-p_i = o(p_i)$.)  Now, pick
$\alpha>0$, depending on $m \text{ and } i$, such that
$$
-p_ia< -\alpha < -\alpha + 2 C_3 (p_{i+m}-p_i) < 0 < 
\alpha < \alpha+ 2 C_3 (p_{i+m}-p_i) < p_i a .
$$
For instance, $\alpha$ satisfying $p_i a - 4 C_3 (p_{i+m}-p_i) <
\alpha < p_i a - 2 C_3 (p_{i+m}-p_i) $ would work. As a
consequence, we also have
$$
-p_{i+m} a < -\alpha + C_3 (p_{i+m}-p_i) < 0 
< \alpha + C_3 (p_{i+m}-p_i) < p_{i+m} a.
$$

Finally, let us specify $\tQ$ and, in turn, $\tH$. To this end, we
choose $\eps>0$ so small that $\| \varphi^{\eps(p_{i+m}-1)}_{Q}\| \leq
2$ and hence \eqref{eq:shift} is satisfied with $k=p_i$ and
$k+l=p_{i+m}$ in the second assertion of Lemma \ref{lemma:shift}.  Set
$\delta:= C_3 (p_{i+m}-p_i)$. Then, for a linear homotopy, which we
may assume to be slow in the sense of Section \ref{sec:homotopy}, from
$\tH^{\nat p_i}$ to $\tH^{\nat p_{i+m}}$, we have the induced map
$$
\HF^{(-\alpha,\,\alpha)}\left( \tH^{\nat p_i} \right)
\to
\HF^{(-\alpha,\,\alpha) + \delta}\left( \tH^{\nat p_{i+m}} \right).
$$
Here the fact that $\delta$ is the correct action shift follows from
Theorem \ref{thm:homotopy} and Lemma \ref{lemma:shift}.  Likewise, the
linear-homotopy map from $ \tH^{\nat p_{i+m}} $ to $\tH^{\nat p_i}$
results in another action shift in $\delta$.  Consider now the following
commutative diagram:
$$
\xymatrix{ 
& \HF^{( -\alpha,\,\alpha) + \delta}_{s_i }\left(\tH^{\nat p_{i+m}}\right)=0
  \ar[rd] \\
  0 \neq \HF^{(-\alpha,\,\alpha)}_{s_i }\left( \tH^{\nat p_i}\right)
   \ar[rr]^\cong \ar[ru]& &
  \HF^{(-\alpha,\,\alpha)+ 2 \delta}_{s_i} \left (\tH^{\nat p_i} \right) 
}
$$
Here the top group is zero due to our choice of the degree $s_i$. On
the other hand,
$$
\HF^{(-\alpha,\,\alpha)}_{s_i }(\tH^{\nat p_i}) 
=
\HF_{s_i }(H^{\nat {p_i}}, x^{p_i}) \oplus \hdots
\neq 0,
$$
and the horizontal arrow is induced by the natural quotient-inclusion
map; see, e.g. \cite{Gi:coiso}. This is, indeed, an isomorphism by the
stability of filtered Floer homology (see, e.g., \cite{GG:gap})
because $0$ is the only action value in the intervals
$(-\alpha,\,\alpha)$ and $(-\alpha,\,\alpha)+2\delta$. To summarize, a
non-zero isomorphism factors through a zero group in the diagram. This
contradiction completes the proof of Theorem \ref{thm:main}, modulo a
proof of Lemma \ref{lemma:Qtilde} which is given below.
\end{proof}

\begin{Remark}  
  Notice that we have actually established the existence of a simple
  periodic orbit of either $H^{\nat p_i}$ or $H^{\nat p_{i+m}}$. In
  particular, starting with a sufficiently large prime number, among
  every $m$ consecutive primes, there exists at least one prime which
  is the period of a simple periodic orbit of $\varphi_H$.

  Furthermore, for an infinite sequence of simple $p_l$-periodic
  orbits $x_l$ of $\varphi_H$ found this way, where $p_l \to \infty$,
  we have $\Delta_{H^{\nat p_l}} (x_l)/ p_l \to \Delta_H (x)$.  Hence,
  in some sense, the mean index $\Delta_H (x)$ is an accumulation
  point in the union of normalized index spectra for $H$ and its all
  iterations. (Of course, $\Delta_H (x)$ could possibly be isolated,
  but then $\Delta_{H^{\nat p_l}} (x_l)/ p_l = \Delta_H (x)$.)  A
  similar fact also holds for the action.

  Finally note that the condition that the eigenvalues $\sigma$ of $Q$
  are real can be relaxed and replaced by the requirement that $\lvert
  \text{Re } \sigma\rvert > \lvert \text{Im } \sigma \rvert $; cf.\
  Remark \ref{rk:complex}.
\end{Remark}

\begin{proof}[Proof of Lemma \ref{lemma:Qtilde}]
  We construct the function $\tQ$ in three steps and then show that
  $\tQ$ has the required properties.
\subsubsection*{Step 1}
Set $c=\sup_V|Q|$. Let $\eta \colon \R \to \R $ be a smooth function
such that 
\begin{itemize}
\item $\eta$ is odd, 
\item $\eta (x) = x$ when $|x| \leq c $,
\item $\eta (x) = \eps x$ when $|x| \geq c'=2c/\eps $,
\item $\eta' \geq \eps/2$.
\end{itemize}
It is easy to see that $\eta$ with these properties exists.  Note that
to have a monotone function $\eta$ such that $\eta(x)=x$ when $|x|\leq
c$ and $\eta(x)=\eps x$ when $|x| \geq c'$, we must have
$c'>c/\eps$. This is the main reason why the radius $R$ in the
statement of the lemma must be of order $1/\sqrt{\eps}$. As the first
step in the construction of $\tQ$, we replace $Q$ by $\eta\circ Q$.

\subsubsection*{Step 2} In the second step, we appropriately cut off
$\eta\circ Q$ and define a new Hamiltonian $\hQ$ which is a linear
transition from $\eta\circ Q$ to $\eps Q$ in the $q$-direction.
Namely, let $r>0$ be the radius of the ball $V$ and set $a_0 = r /
\sqrt{\eps}$ and $a_1 = 2 a_0$. (So, $r < a_0 < a_1$.) The
modification of $\eta\circ Q$ takes place on the domain $a_0 \leq
\|q\| \leq a_1$. To this end, choose a smooth monotone increasing
function $\phi \colon [0,\infty) \to \R $ such that $\phi (x) = 0$
when $x \leq a_0$ and $\phi(x) = 1$ when $x \geq a_1$, and $|\phi'|
\leq 2/|a_1-a_0|=2/a_0$. We then set
$$
\hQ = \phi(\|q\|) \eps Q + \big( 1 - \phi(\|q\|) \big) \eta\circ Q.
$$

\subsubsection*{Step 3}
In the third step, we suitably cut off $\hQ$ and finally define the
desired Hamiltonian $\tQ$. To this end, let $b_0= \max \{ r, 32c /
\lambda r \sqrt{\eps} \}$ and $b_1=2 b_0$. Here, as in Section
\ref{sec:Floer}, $ \lambda = \min \lambda_i>0$, where $\lambda_i$'s
are the eigenvalues of $A$. (The reason for this choice of $b_0$ will
be clear at the end of the proof.)  Choose a smooth monotone
increasing function $\psi \colon [0,\infty) \to \R $ such that $\psi
(x) = 0$ when $x \leq b_0$ and $\psi (x) = 1$ when $x \geq
b_1$. Define
$$
\tQ = \psi(\|p\|) \eps Q + \big( 1 - \psi(\|p\|) \big) \hQ.
$$

\subsubsection*{Checking the conditions (i)-(iv)}
Since $\eta\circ Q=Q$ on $V$ by the definition of $c$, and $a_0 \geq
r$ and $b_0 \geq r$, we clearly have $\tQ = Q$ on the ball
$V$. Furthermore, $\tQ=\eps Q$ outside the ball of radius
$R=\sqrt{a_1^2+b_1^2}$. Let $V_\eps$ be this ball. It is clear from
our choice of $a_1$ and $b_1$ that $R$ has the form $C_1/\sqrt{\eps}$,
where $C_1$ is independent of $\eps$. This proves (i) and (ii).

To establish (iii), observe first that
\begin{equation}
\label{eq:4c}
\sup |\eta(x) - \eps x|
=
\sup_{[0,\,c']} |\eta(x) - \eps x| 
\leq
\eta(c') + \eps c' 
=
2 \eps c'
=
4 \eps c/\eps  
=
4 c.
\end{equation}
Thus
\begin{equation*}
\begin{split}
\sup_{V_\eps}|\tQ| 
& \leq
\sup_{V_\eps}|\eps Q| + \sup_{V_\eps}|\hQ| \\
& \leq 
2\sup_{V_\eps}|\eps Q| + \sup_{V_\eps} |\eta\circ Q| \\
& \leq 
3\sup_{V_\eps}|\eps Q| + \sup_{V_\eps} |\eta\circ Q - \eps Q| \\
& \leq
3\eps \frac{C_1^2}{\eps}\sup_{B(1)}|Q| + 4c\\
& =
3C_1^2 \sup_{B(1)}|Q| + 4c =: C_2,
\end{split}
\end{equation*}
with $C_2$ independent of $\eps$.  This is where replacing $Q$ by
$\eta \circ Q$ in Step 1 is essential.

To verify the condition (iv), note that without loss of generality we
may assume that the off-diagonal part of $A$ is so small that
$$
\pounds_{X_Q}\|p\|^2\leq -\lambda\|p\|^2 
\text{ and }
\pounds_{X_Q}\|q\|^2\leq -\lambda\|q\|^2 .
$$
(Here we dropped the factor of $2$ on the right hand side of the
inequalities to account for the off-diagonal terms.) In particular,
every integral curve of $Q$ enters the polyball $P$ through the
``side'' part, $\|p\|=\const$, of the boundary $\p P$ and leaves it
through the ``top'', $\|q\|=\const$, of $\p P$.

We will show that
\begin{itemize}
\item[(a)] the flow of $\tQ$ is equal to the flow of $Q$ on $P$ and on
  the disk $q=0$, $\|p\|\leq b_0$,
\item[(b)] $\pounds_{X_{\tQ}}\|q\|^2\geq 0$ when $\| p\|\leq b_0$, with
  strict inequality when $q\neq 0$,
\item[(c)] $\pounds_{X_{\tQ}}\|p\|^2< 0$ when $\| p\|\geq b_0$.
\end{itemize}
It is not hard to see that (iv) readily follows from these assertions. 

The assertion (a) is obvious since $\tQ=Q$ in the region where
$|Q|\leq c$ and $\|q\|\leq a_0$ and $\|p\|\leq b_0$, containing
$V\supset P$ and the disk $q=0$, $\|p\|\leq b_0$. It remains to check
(b) and (c), i.e., that the function $\|q\|^2$ increases along the
flow of $\tQ$ when $\|p\| \leq b_0$, and the function $\|p\|^2$
decreases along the flow of $\tQ$ when $\|p\| \geq b_0$, unless $q =
0$.

To prove (b), first note that $\tQ=\hQ$ in the region where $\|p\|\leq
b_0$. Also, recall that $\eta'\geq \eps/2$ and, hence,
$$
\eps \phi(\|q\|) + \big( 1 - \phi(\|q\|) \big) (\eta'\circ Q ) \geq \eps/2.
$$
Therefore, since $\phi(\|q\|)$ is independent of $p$, we have
\begin{equation*}
\begin{split}
\pounds_{X_{\tQ}}\|q\|^2 &=\pounds_{X_{\hQ}}\|q\|^2\\
& =
\eps\phi(\|p\|)\pounds_{X_{Q}}\|q\|^2
+\big(1-\phi(\|p\|)\big)(\eta'\circ Q ) \pounds_{X_{Q}}\|q\|^2\\
& \geq 
\frac{\eps\lambda}{2} \|q\|^2 \geq 0.
\end{split}
\end{equation*}

Let us now establish (c), which is somewhat more involved than (b) due
to the $q$-dependence of $\phi$. Writing $\psi$ for $\psi(\|p\|)$ and
$\phi$ for $\phi(\|q\|)$ and $\phi'$ for $\phi'(\|q\|)$, we have
\begin{equation*}
\begin{split}
\pounds_{X_{\tQ}}\|p\|^2 
& = 
\eps\psi\pounds_{X_{Q}}\|p\|^2
+(1-\psi) \pounds_{X_{\hQ}}\|p\|^2\\
& =
\eps\psi\pounds_{X_{Q}}\|p\|^2 \\
& \quad +(1-\psi)\Big(\eps \phi\pounds_{X_{Q}}\|p\|^2
+ (1-\phi) (\eta'\circ Q) \pounds_{X_{Q}}\|p\|^2\Big)\\
& \quad + (1-\psi)\big(\eps Q - \eta\circ Q\big)
\phi'\pounds_{X_{\|q\|}}\|p\|^2\\
& = 
\Big( \eps\psi + (1-\psi) \big(\eps \phi + (1-\phi) (\eta'\circ Q)\big) \Big) 
\pounds_{X_{Q}}\|p\|^2 \\
& \quad + (1-\psi)
\big(\eps Q - \eta\circ Q\big) \phi'\pounds_{X_{\|q\|}}\|p\|^2.
\end{split}
\end{equation*}
To bound from above the first term in this expression, note that the
coefficient of $\pounds_{X_{Q}}\|p\|^2$ is positive and satisfies
$$
\eps\psi + (1-\psi) \big(\eps \phi + (1-\phi) (\eta'\circ Q)\big)
\geq 
\eps\psi 
+ \frac{\eps}{2}(1-\psi) \\
=
\frac{\eps}{2}(1+\psi)
\geq 
\frac{\eps}{2}.
$$
Therefore, 
$$
\Big( \eps\psi + (1-\psi) \big(\eps \phi 
+ (1-\phi) (\eta'\circ Q)\big) \Big) \pounds_{X_{Q}}\|p\|^2
\leq 
- \eps\frac{\lambda}{2} \|p\|^2.
$$
Bounding from above the second term, our choices of $ a_0 = r /
\sqrt{\eps}$ and $a_1=2 a_0$ and also the requirement that $|\phi'|
\leq 2/|a_1-a_0|=2/a_0$ enter the picture. Then, by \eqref{eq:4c}, we
have
$$
2 |\eps Q - \eta\circ Q| \cdot |\phi'|  \cdot \|p\|
\leq
8c \frac{2}{|a_1-a_0|}\|p\| 
=
\frac{16c}{a_0}\|p\| 
\leq
\frac{16c \sqrt{\eps}}{r} \|p\|.
$$
Thus, in the region where $\|p\| > b_0 = 32c / \lambda r \sqrt{\eps}
$, we have
$$
\pounds_{X_{\tQ}}\|p\|^2 
< 
- \eps\frac{\lambda}{2} \|p\|^2
+
\frac{16c \sqrt{\eps}}{r} \|p\|
<0,
$$
which completes the proof of (c). (Note that $b_0$ is chosen exactly
to make (c) hold.)
\end{proof}

\subsection{Proof of Theorem \ref{thm:lowdim}}
\label{sec:proof-lowdim}
As has been pointed out in the introduction, we have a more general
result in dimension two:

\begin{Theorem}
\label{thm:R2}
Let $H \colon S^1 \times \R^2 \to \R$ be a Hamiltonian which is equal
to a hyperbolic quadratic form at infinity. Assume that $\varphi_H$
has at least two isolated homologically non-trivial fixed points and
$\Fix(\varphi_H)$ is finite.  Then $\varphi_H$ has simple periodic
orbits of arbitrarily large period.
\end{Theorem}

\begin{Remark}
\label{rk:elliptic}
  A similar two-dimensional result holds when $H$ is elliptic
  quadratic at infinity. (In fact, the requirement that the fixed
  points be homologically non-trivial is not needed in this case.)
  Indeed, since the Hamiltonian is elliptic outside a compact set, a
  sufficiently large sub-level will be invariant under the flow. Then
  Franks' theorem stating that an area-preserving map of the two-disk
  has either one or infinitely many periodic points, \cite{Fr1},
  implies the result.
\end{Remark}

\begin{Remark} 
\label{rk:Franks}
Using Theorem \ref{thm:R2}, we can also prove a weaker version of
Franks' theorem on $S^2$, asserting that a Hamiltonian diffeomorphism
of $S^2$ with a hyperbolic fixed point must necessarily have
infinitely many periodic orbits. However, the argument is somewhat
involved, and we omit it since this result also follows from the main
theorem of \cite{GG:hyperbolic} and, as has been mentioned in the
introduction, at least two other symplectic proofs of Franks' theorem
are available, \cite{BH,CKRTZ,Ke:JMD}.
\end{Remark}

\begin{proof}[Proof of Theorem \ref{thm:R2}]
  Observe that if $\varphi_H$ has a homologically non-trivial fixed
  point with non-zero mean index, then the theorem follows from
  Theorem \ref{thm:main}. So, let us assume that there are at least
  two isolated homologically non-trivial fixed points with zero mean
  index. Notice that all of these points cannot have non-zero local
  Floer homology concentrated in degree zero: $\HF_*(H)=0$ when
  $*\neq0$ and $\HF_0(H)=\Z_2$.  Thus $\varphi_H$ must have at least
  one fixed point with zero mean index and non-zero local Floer
  homology in degree $\pm 1$. Such an orbit is a symplectically
  degenerate maximum, and its presence implies that $\varphi_H$ has
  simple periodic orbits of arbitrarily large prime period; see
  \cite{GG:gaps,GG:gap} and also \cite{Gi:conley,He:irr,Hi}. (Strictly
  speaking, the latter fact has been established only for (a broad
  class of) closed symplectic manifolds. However, since the
  Hamiltonian in our case is a compactly supported perturbation of a
  hyperbolic quadratic form on $\R^{2n}$, having periodic orbits only
  within the support of the perturbation, the proof in the case of
  closed manifolds, for instance the one in \cite{GG:gaps}, goes
  through word-for-word.)
\end{proof}

\begin{proof}[Proof of Theorem \ref{thm:lowdim} in dimension four]
  Recall that the homology is concentrated in degree zero and
  $\HF_0(H)=\Z_2$. Hence one of the fixed points of $\varphi_H$ must
  have non-zero Conley-Zehnder index. By a straightforward index
  analysis, it is easy to see that, in dimension four, such an orbit
  must necessarily have non-zero mean index. (Indeed, observe that, in
  dimension four, the mean index of a non-degenerate fixed point is
  zero if and only if the linearization is hyperbolic or its
  eigenvalues comprise two pairs ``conjugate'' to each other. It is
  clear that in both cases the Conley-Zehnder index is zero.) Finally,
  applying Theorem \ref{thm:main}, we obtain the existence of simple
  periodic orbits with arbitrarily large prime period.
\end{proof}

\end{document}